\documentclass[11pt,bezier]{article}
\usepackage{amsmath,amssymb,amsfonts,euscript,graphicx,algorithm,algorithmic,makeidx,amscd,enumerate,setspace,tkz-berge, amscd}

\newtheorem{prethm}{{\bf Theorem}}

\newenvironment{thm}{\begin{prethm}{\hspace{-0.5
em}{\bf.}}}{\end{prethm}}

\newtheorem{prepro}[prethm]{{\bf Theorem}}

\newtheorem{preprop}[prethm]{{\bf Proposition}}

\newenvironment{prop}{\begin{preprop}{\hspace{-0.5
em}{\bf.}}}{\end{preprop}}

\newtheorem{precor}[prethm]{{\bf Corollary}}

\newenvironment{cor}{\begin{precor}{\hspace{-0.5
em}{\bf.}}}{\end{precor}}

\newtheorem{predefinition}[prethm]{{\bf Definition}}

\newenvironment{definition}{\begin{predefinition}{\hspace{-0.5
em}{\bf.}}}{\end{predefinition}}

\newtheorem{preconj}[prethm]{{\bf Conjecture}}

\newtheorem{preremark}[prethm]{{\bf Remark}}

\newenvironment{remark}{\begin{preremark}\rm{\hspace{-0.5
em}{\bf.}}}{\end{preremark}}

\newtheorem{preexample}[prethm]{{\bf Example}}

\newenvironment{example}{\begin{preexample}\rm{\hspace{-0.5
em}{\bf.}}}{\end{preexample}}

\newtheorem{prelem}[prethm]{{\bf Lemma}}

\newenvironment{lem}{\begin{prelem}{\hspace{-0.5
em}{\bf.}}}{\end{prelem}}

\newtheorem{prelam}{{\bf Lemma}}

\newtheorem{preproof}{{\bf Proof.}}

\newenvironment{proof}[1]{\begin{preproof}{\rm
#1}\hfill{$\Box$}}{\end{preproof}}

\title{\bf \large GORENSTEIN $\pi[T]$-PROJECTIVITY WITH RESPECT TO A TILTING MODULE
\thanks
{{\it Key Words}: Coherent, Dimension, Gorenstein, Tilting}
\thanks {2010{ \it Mathematics Subject Classification}: 13D07; 16D40; 18G25;}}

\author{{\normalsize M. Amini${}$  {}}\vspace{3mm}\\
{\footnotesize{${}^{\mathsf{}}$\it }}\\
{\footnotesize{${}^{\mathsf{}}$\it Department of Mathematics, Payame Noor University, Tehran, Iran.}}\\
{\footnotesize{}}\\
{\footnotesize{}}
{\footnotesize{$\mathsf{}$\quad\quad
$\mathsf{ mostafa.amini@pnu.ac.ir}$\quad\quad $\mathsf{}$}}}

\date{}

\begin{document}
\maketitle
%\begin{abstract}
{\small\noindent �{\bf{Abstract.}} Let $T$ be a tilting module. In this paper, Gorenstein $\pi[T]$-projective modules are introduced and some of their basic properties are studied. Moreover, some characterizations of rings over which all modules are Gorenstein $\pi[T]$-projective are given. For instance, on the $T$-cocoherent rings, it is proved that the Gorenstein $\pi[T]$-projectivity of all $R$-modules is equivalent to the $\pi[T]$-projectivity of $\sigma[T]$-injective as a module.
%\end{abstract}

%\begin{center}
\vspace{5mm} \noindent{\bf\large 1. Basic Definitions and Notations}\\\\
%\end{center}
Throughout this paper,
$R$ is an associative ring with non-zero identity, all modules are
unitary left $R$-modules. First we recall some known notions and facts needed in the sequel. Let $R$ be a ring and $T$ an R-module. Then 
\begin{enumerate}
\item[(1)] We denote by
${Prod}T$ (resp.
${F.Prod}T$), the class of modules isomorphic to direct
summands of direct product of copies (resp. finitely many copies) of $T$.

\item[(2)] We denote by
${Add}T$ (resp.
${F.Add}T$), the class of modules isomorphic to direct
summands of direct sum of copies (resp. finitely many copies) of $T$.

\item[(3)]  Following \cite{qv}, a module $T$ is called tilting ($1$-tilting) if it satisfies the following conditions:\\
(a) pd$(T)\leq 1$, where $pd(T)$ denotes the {\it projective dimension of $T$}.\\
(b) Ext$^{i}(T,T^{(\lambda)})=0$, for each $i>0$ and for every cardinal $\lambda$.\\
(c) There exists the exact sequence $0\rightarrow R\rightarrow
T_{0}\rightarrow T_{1}\rightarrow 0$, where $T_0,T_1\in { \rm Add}T$.

\item[(4)] By
$Copres^{n}T$ (resp. $F.Copres^{n}T$) and $Copres^{\infty}T$ (resp. $F.Copres^{\infty}T$), we denote the set of all modules $M$ such
that there exists exact sequences $$\begin{array}{cccccccccccc}
\vspace{-.10cm} 0 \longrightarrow M\longrightarrow
T_{0}\longrightarrow T_{1}\longrightarrow\cdots\longrightarrow T_{n-1}\longrightarrow T_{n}
\end{array}$$
and
$$\begin{array}{cccccccccccc} \vspace{-.10cm}
0 \longrightarrow M\longrightarrow
T_{0}\longrightarrow T_{1}\longrightarrow\cdots\longrightarrow T_{n-1}\longrightarrow T_{n}\longrightarrow\cdots
,\end{array}$$ respectively, where $T_{i}\in {\rm Prod}T$ (resp. $T_{i}\in{\rm F.Prod}T$), for
every $i\geq 0$. 

\item[(5)] A module $M$ is said to be \textit{cogenerated},
by $T$, denoted by $M\in {Cogen}T$,
(resp. \textit{generated}, denoted $M\in{Gen}T$) by $T$
if there exists an exact sequence $0\rightarrow M\rightarrow T^n$ (resp. $ T^{(n)}\rightarrow M\rightarrow 0 $ )
, for some positive integer $n$.

\item[(6)]  Let
$\mathcal{C}$ be a class of modules and $M$ be a module. A {\it right (resp. left) $\mathcal{C}$-resolution} of $M$ is a long
exact sequence $0\rightarrow M\rightarrow C_0\rightarrow C_1\rightarrow \cdots$ (resp. $ \cdots\rightarrow C_1\rightarrow C_0\rightarrow M\rightarrow 0$), where $C_i\in \mathcal{C}$, for all
$i\geq 0$. It is said that a module $M$ has right
{\it $\mathcal{C}$-dimension} $n$ (briefly, ${\mathcal{C}}.dim(M)=n$) if
$n$ is the least non-negative integer such that there exists a
long exact sequence
$$0\longrightarrow M\longrightarrow C_{0}\longrightarrow C_{1}\longrightarrow
\cdots\longrightarrow C_{n-1}\longrightarrow C_{n}\longrightarrow 
0$$ with $C_{i}\in \mathcal{C}$, for each $i\geq 0$. In particular, the ${\rm Prod}T$-dimension of $M$ is called {\it$T$-injetive dimension} of $M$ and is denoted by $T.i.dim(M)$. Note that for any tilting module $M$, if $M\in {\rm Cogen}T$, then \cite[Proposition 2.1]{shaveisicam} implies that ${\rm Cogen}T={\rm Copres}^\infty T$.
This shows that any module cogenerated by $T$ has an ${\rm Prod}T$-resolution.
The ${\rm Prod}T$-resolutions and the relative homological dimension were studied by Nikmehr and Shaveisi in \cite{shaveisicam}. 

\item[(7)] For any homomorphism $f$, we denote by ${ker}f$ and ${im}f$, the kernel and image of $f$, respectively. Let $A$ and $M\in {\rm
Cogen}T$ be two modules. We define the functor $$ {\mathcal{E}}_{T}^{n}(A,M):= \frac{{\rm ker}\delta_{*}^{n}}{{\rm
im} \delta_{*}^{n-1}},$$ where
$$\begin{array}{cccccccccc} \vspace{-.3cm}
&&&\delta_{0}&&\delta_{1}&&\delta_{n}\\
0& \longrightarrow&M&\longrightarrow&T_{0}& \longrightarrow &  \cdots
&\longrightarrow & T_{n} & \longrightarrow  \cdots
\end{array}$$

${\rm Prod}T$-resolution of $M$ and $\delta_{*} ^{n}={\rm Hom}(id_B,\delta_n)$,
for every
$i\geq 0$.
 see \cite{shaveisicam, amini} for more details.

\item[(8)]  Let $M\in {\rm Cogen}T$ and $N$ be two modules. A similar proof to that of \cite[Lemma 2.11]{seam1} shows that ${\mathcal{E}}^0_T(N,M)\cong {\rm Hom}(N,M)$. Moreover, ${\mathcal{E}}^1_T(-,M)=0$ implies that $M\in {\rm Prod}T$, and if $M\in {\rm Gen}T$, then  ${\mathcal{E}}^1_T(M,-)=0$ implies that $M\in {\rm Add}T$.
It is clear that ${\rm T.i.dim}(M)=n$ if and only if $n$ is the least non-negative integer such that
${\mathcal{E}}^{n+1}_{T}(A,M)=0$, for
any module $A$, see \cite[Remark 2.2]{shaveisicam} for more details. So, ${\rm T.i.dim}(M) = n$ if and only if ${\mathcal{E}}_{T}^{n+i}(A,M)=0$ for every module $A$ and every $i\geq 1$.
 A module with zero $T$-injective dimension (resp. $T$-projective dimension) is called
{\it $T$-injective {\rm(}resp. $T$-projective{\rm)}}.
%We refer the reader to \cite{shaveisicam} for more details.
A similar proof to that of \cite[Proposition 2.3]{seam1} shows
that the definition of ${\mathcal{E}}_{T}^{n}(C,M)$  is independent from the choice of
 ${\rm Prod}T$-resolutions.  For unexplained concepts and notations, we refer the reader to \cite{fullerbook,shaveisicam,rotman}.

\item[(9)] For module $T$, we denote by $\pi[T]$, the full subcategory of modules whose
objects are of the form $\frac{B}{A}\leq \frac{T^I}{A}$,
 for
some cardinal $I$ and some modules $A \leq B \leq T^I$.  Also,  the full subcategory $\sigma[T]$ of modules subgenerated by a given module $T$ (see \cite{W.R}). 
\item [(10)] 
$G$ is called  Gorenstein $\sigma[T]$-injective  if there exists an exact sequence of $\sigma[T]$-injective modules
$${\mathbf{A}}= \cdots\longrightarrow A_1\longrightarrow A_{0}\longrightarrow A^0\longrightarrow
A^1\longrightarrow\cdots$$ with $G={\rm ker}(A^0\rightarrow A^1)$ such that ${\rm Hom}(U,-)$ leaves this sequence exact whenever $U\in {\rm Pres}^{1}T$ with ${\rm T.p.dim}(U)<\infty$ (see \cite{amini}).
\item [(11)] $M$ is said to be {\it  finitely cogenerated} \cite{fullerbook}  if for every family $\{V_{k}\}_{J}$ of submodules of $M$ with ${\bigcap}_{J} V_{k}=0$, there is a
finite subset $I\subset J$ with ${\bigcap}_{I} V_{k}=0$.
\item [(12)] $M$ is said to be {\it finitely copresented}  if there is an exact sequence of
 $R$-modules $0\rightarrow M\rightarrow E^0 \rightarrow E^1$, where each $E^i$ is a
finitely cogenerated injective module, see \cite{ M.F, E.W, Z.M.C}. 
 \end{enumerate}
Let $T$ be a tilting module. In this paper, we introduce the {\it $\pi[T]$-projective} modules, {\it the $\pi[T]$-projective dimension} and {\it Gorenstein $\pi[T]$-projective} modules. 

Let $M\in {\rm Gen}T$. Then, $M$ is called $\pi[T]$-projective if the functor ${\mathcal{E}}_{T}^{1}(M,-)$ vanishes on $\pi[T]$.
Also,  the $\pi[T]$-projective dimension of $M$ is defined to be
$$\pi[T].pd(M)=\inf\{n:\ \mathcal{E}^{n+1}_T(M,N)=0\ {\rm for \ every} \ N\in\pi[T]\}.$$
We define a module $G$ to be  Gorenstein $\pi[T]$-projective ( $GT$-projective
for short), if there exists an exact sequence of $\pi[T]$-projective modules
$${\mathbf{B}}= \cdots\longrightarrow B_1\longrightarrow B_{0}\longrightarrow B^0\longrightarrow
B^1\longrightarrow\cdots$$ with $G={\rm ker}(B^0\rightarrow B^1)$ such that ${\rm Hom}(-,U)$ leaves this sequence exact whenever $U\in {\rm F.Copres}^{1}T$ with ${\rm T.i.dim}(U)<\infty$. In this paper, the $GT$-projective dimension of a module $G$ is denoted by $GT$-$pd(G)$.

In Section 2, we study some basic properties of the Gorenstein $\pi[T]$-projective modules. Recall that a ring $R$ is said to be {\it cocoherent} if every finitely cogenerated module is finitely copresented. So, $R$ is a cocoherent ring if and only if ${\rm Copres}^0R={\rm Copres}^1R$. For more information about the cocoherent rings, we refer the reader to \cite{Glaz 1989}. As a cogeneralization of this concept, we call a ring $R$ to be {\it $T$-cocoherent} if ${\rm F.Copres}^0T={\rm F.Copres}^1T$.

Section 3 is devoted to some characterizations of $T$-cocoherent
rings over which all modules are Gorenstein $\pi[T]$-projective. For instance, it is proved that every module is Gorenstein $\pi[T]$-projective if and if every $T$-injective module is $\pi[T]$-projective if and if every $\sigma[T]$-injective module is Gorenstein $\pi[T]$-projective. Finally, we give a sufficient condition under which every Gorenstein $\pi[T]$-projective module is $\pi[T]$-projective.

\section{Gorenstein $\pi[T]$-Projectivity}
We start with the following definition.
\begin{definition}\label{2.76}
Let $T$ be a tilting module. Then
\begin{enumerate}
\item [\rm (1)]
$M$ is called $\pi[T]$-projective if  ${\mathcal{E}}_{T}^{1}(M,N)=0$, for every $N\in\pi[T]$.
\item [\rm (2)]
  Let $G\in {\rm Gen}T$. Then, $G$ is called  Gorenstein $\pi[T]$-projective if there exists an exact sequence of $\pi[T]$-projective modules
$${\mathbf{B}}= \cdots\longrightarrow B_1\longrightarrow B_{0}\longrightarrow B^0\longrightarrow
B^1\longrightarrow\cdots$$ with $G={\rm ker}(B^0\rightarrow B^1)$ such that ${\rm Hom}(-,U)$ leaves this sequence exact whenever $U\in {\rm F.Copres}^{1}T$ with ${\rm T.i.dim}(U)<\infty$
\end{enumerate}
\end{definition}
\begin{remark}\label{2.g}
Let $T$ be a tilting module. Then
\begin{enumerate}
\item [\rm (1)]
 ${\mathcal{E}}_{T}^1(N,M)=0$ for any $\pi[T]$-projective module $N$ and any $M\in{\rm Copres}^0T$.
\item [\rm (2)]
If $A\in{\rm Add}T$, then $A$ is $\pi[T]$-projective.
\end{enumerate}
\end{remark}

\begin{lem}\label{2.56}
Let $ 0\rightarrow A\stackrel{\displaystyle f}\rightarrow B\stackrel{\displaystyle g}\rightarrow C\rightarrow 0$ be an exact sequence. Then
\begin{enumerate}
\item [\rm (1)]
If $A$ is $T$-injective and $A, B, C \in {\rm Cogen}T$, then $B=A\oplus C$.
\item [\rm (2)]
If $A\in{\rm F.Copres}^{n}T$ and $C\in{\rm F.Copres}^{n}T$, then $B\in{\rm F.Copres}^{n}T.$
\item [\rm (3)]
If $C\in{\rm F.Copres}^{n}T$ and $B\in{\rm F.Copres}^{n+1}T$, then $A\in{\rm F.Copres}^{n+1}T.$
\item [\rm (4)]
If $B\in{\rm F.Copres}^{n}T$ and $A\in{\rm F.Copres}^{n+1}T$, then $C\in{\rm F.Copres}^{n}T.$

\end{enumerate}
\end{lem}
\begin{proof}
(1) 
If $A$ is $T$-injective and $A, B, C \in {\rm Cogen}T$, then we deduce that the
sequence
$$ 0\longrightarrow {\rm Hom}(C,A)\stackrel{\displaystyle g^{*}}\longrightarrow {\rm Hom}(B,A)\stackrel{\displaystyle f^{*}}\longrightarrow
{\rm Hom}(A,A)\longrightarrow {\mathcal{E}}_{T}^{1}(C,A)=0$$
is exact. So, 
there exists $h: B\rightarrow A$ such that $hf=1_{A}.$

(2)
We prove the assertion by induction on $n$. If $n=0$, then the commutative
diagram with exact rows

$$
 \begin{array}{ccccccccccccccccc}
& \hspace{-3.7cm}0&\hspace{-3.7cm}0 &\hspace{-2.7cm}0& \\
 \hspace{1cm} \downarrow  & \hspace{0.7cm}\downarrow &\hspace{-0.2cm}\downarrow & \\
  o\longrightarrow A& \stackrel{\displaystyle f}\longrightarrow B&\stackrel{\displaystyle g}\longrightarrow C\longrightarrow 0 \\
      \hspace{1.5cm} \downarrow h_{0}^{'}&\ \ \ \ \ \ \ \ \ \ \downarrow h_{0}&\ \  \downarrow  h_{0}^{''} &  \\
    0\longrightarrow   T^{'}_0 &{\hspace{-.6cm}\stackrel{\displaystyle \i_0} \longrightarrow} \ \ \ \ \  T^{'}_0\oplus T^{''}_0&{\hspace{-.3cm}\stackrel{\displaystyle \pi_0}\longrightarrow} \ \ T^{''}_0\longrightarrow 0\\
 \hspace{1cm}  \downarrow  & \ \ \ \ \ \downarrow & {\hspace{-4mm}\downarrow} &  \\

     \end{array}
$$

\noindent exists, where $T'_0,T''_0\in {\rm F.Prod}T$, $i_0$ is the inclusion map, $\pi_0$ is a canonical epimorphism and $h_0=i_{0}h_0'$ is endomorphism, by Five Lemma. Let $K_1'={\rm coker} (h_0')$, $K_1={\rm coker}(h_0)$ and $K_1''={\rm coker}( h_0'')$. It is clear that $(T^{'}_0\oplus T^{''})\in{\rm F.Prod}T$ and $K_1',K_1''\in{\rm F.Copres}^{n-1}T$; so, the induction implies that $K_1\in {\rm F.Copres}^{n-1}T$. Hence $B\in {\rm F.Copres}^{n}T$.

(3) Let $B\in{\rm F.Pres}^{n+1}T$ and $C\in{\rm F.Pres}^nT$,  then the following commutative diagram with exact rows:
\begin{center}
$
 \begin{array}{ccccccccc}
&{\hspace{-7cm}0}&{\hspace{-5.8cm}0}&  \\
    &{\hspace{-7cm}\downarrow} &{\hspace{-5.8cm}\downarrow}&  \\
  \hspace{-3cm}0\longrightarrow A&\hspace{-5.7cm}= {\hspace{-1.8mm}=}  A&&&  \\
    \hspace{-2cm}\downarrow & \hspace{-5cm}\downarrow &  \\
    0\longrightarrow B \longrightarrow T_{0}\longrightarrow L\longrightarrow 0 \\
    \hspace{-2cm}\downarrow &  &\hspace{-5.5cm}\downarrow &\hspace{-4.3cm}\parallel  &  \\
      0\longrightarrow C\longrightarrow D\longrightarrow L\longrightarrow 0& \\
 &&\hspace{-7.8cm}\downarrow &\hspace{-6.3cm}\downarrow&  \\
  & &  \hspace{-7.8cm}  0&\hspace{-6.3cm}  0
     \end{array}
$
\end{center}
\noindent where $T_{0} \in {\rm F.Prod}T$ and $L\in {\rm F.Copres}^{n}T$. By (2), $D\in {\rm F.Copres}^{n}T$. So, we deduce that $A\in {\rm F.Copres}^{n+1}T$.

(4)  Let $A\in{\rm F.Pres}^{n+1}T$ and $B\in{\rm F.Pres}^nT$,  then the following commutative diagram with exact rows:
\begin{center}
$
 \begin{array}{ccccccccc}
&{\hspace{-7cm}0}&{\hspace{-6cm}0}&  \\
    &{\hspace{-7cm}\downarrow} &{\hspace{-6cm}\downarrow}&  \\
  0\longrightarrow A\longrightarrow T_{0}^{'}\longrightarrow L^{'}\longrightarrow 0   \\
    \hspace{-2cm}\downarrow & \hspace{-5cm}\downarrow &  \\
    0\longrightarrow B \longrightarrow T_{0}\longrightarrow L\longrightarrow 0 \\
    \hspace{-2cm}\downarrow &  &\hspace{-5.5cm}\downarrow &\hspace{-4.3cm}\parallel    &  \\
      0\longrightarrow C\longrightarrow D\longrightarrow L\longrightarrow 0& \\
 &&\hspace{-7.8cm}\downarrow &\hspace{-6.3cm}\downarrow&  \\
  & &  \hspace{-7.8cm}  0&\hspace{-6.3cm}  0
     \end{array}
$
\end{center}
\noindent where $T_{0}, T_{0}^{'}  \in {\rm F.Prod}T$ and $L\in {\rm F.Copres}^{n-1}T$. Since $T_{0}^{'}$ is $T$-injective, we have that $T_{0}=T_{0}^{'}\oplus D$ By (1), and $D\in{\rm Cogen}T$. Thus for any $N\in {\rm Cogen}T$, we have
$$ {\mathcal{E}}_{T}^{1}(T_{0},N)={\mathcal{E}}_{T}^{1}(T_{0}^{'}\oplus D,N)={\mathcal{E}}_{T}^{1}(T_{0}^{'},N)\oplus {\mathcal{E}}_{T}^{1}( D,N)=0.$$ 
Hence $D\in {\rm F.Prod}T$. On the other hand, $L\in {\rm F.Copres}^{n-1}T$. Therefore, we conclude that $C\in {\rm F.Copres}^{n}T.$
\end{proof}

In the following theorem, we show that in the case of $T$-cocoherent rings, the existence
of $\pi[T ]$-projective complex of a module is sufficient to be Gorenstein $\pi[T ]$-projective.
\begin{thm}\label{2.30}
Let $R$ be a $T$-cocoherent ring and $G\in {\rm Gen}T$ be a module. Then $G$ is
Gorenstein $\pi[T]$-projective if and only if there is an exact sequence
$${\mathbf{B}}= \cdots\longrightarrow B_1\longrightarrow B_{0}\longrightarrow B^0\longrightarrow
B^1\longrightarrow\cdots$$
of $\pi[T]$-projective modules such that $G=\ker(B^0\rightarrow B^1)$.
\end{thm}
\begin{proof}
{($\Longrightarrow$) : This is a direct consequence of definition.

($\Longleftarrow$) : By definition, it suffices to show that ${\rm Hom}({\mathbf{B}},U)$ is exact for
every module $U\in {\rm F.Copres}^{1}T$  with ${\rm T.i.dim} (U) = m < \infty$. To prove this, we use the induction on $m$. The case $m =0$ is clear. Assume that $m \geq 1$. Since $U\in {\rm F.Copres}^{1}T$, there exists an exact sequence
$0\rightarrow U\rightarrow T_{0}\rightarrow I \rightarrow 0$ with $T_{0}\in {\rm F.Prod}T\subseteq{\rm F.Copres} ^{0}T$. Now, from the $T$-cocoherence of $R$ and Lemma \ref{2.56}, we deduce that $I, T_{0}\in {\rm F.Copres} ^{1}T$. Also, ${\rm T.i.dim}(I) \leq m-1$ and ${\rm T.i.dim} (T_0) =0$. Thus by Remark \ref{2.g}, the following short exact sequence of complexes exists:
\begin{center}
$
\begin{array}{ccccccccc}
& \vdots&\vdots &\vdots&\\
& \downarrow & \downarrow &\downarrow & \\
0 \longrightarrow &{\rm Hom}(B^1,U) &\longrightarrow {\rm Hom}(B^1,T_{0})&\longrightarrow {\rm Hom}(B^1,I)\longrightarrow 0 \\
& \downarrow & \downarrow &\downarrow & \\
0 \longrightarrow &{\rm Hom}(B^0,U) & \longrightarrow {\rm Hom}(B^0,T_0)&\longrightarrow {\rm Hom}(B^0,I)\longrightarrow 0 \\
& \downarrow &\downarrow &\downarrow & \\
0 \longrightarrow &{\rm Hom}(B_0,U) & \longrightarrow {\rm Hom}(B_0,T_0)&\longrightarrow {\rm Hom}(B_0,I)\longrightarrow 0 \\
& \downarrow &\downarrow &\downarrow & \\
0 \longrightarrow &{\rm Hom}(B_1,U) & \longrightarrow {\rm Hom}(B_1,T_0)&\longrightarrow {\rm Hom}(B_1,I)\longrightarrow 0 \\
& \downarrow &\downarrow &\downarrow & \\
& \vdots&\vdots &\vdots&\\
& \parallel & \parallel &\parallel& \\
0 \longrightarrow &{\rm Hom}({\mathbf{B}},U) & \longrightarrow {\rm Hom}({\mathbf{B}},T_0)&\longrightarrow {\rm Hom}({\mathbf{B}},I)\longrightarrow 0. \\

\end{array}
$
\end{center}

\noindent By induction, ${\rm Hom}({\mathbf{B}},T_{0})$ and ${\rm Hom}({\mathbf{B}},I)$ are exact, hence ${\rm Hom}({\mathbf{B}},U)$ is exact by \cite[Theorem 6.10]{rotman}. Therefore, $G$ is Gorenstein $\pi[T]$-projective.}
\end{proof}

It is worthy to mention that the notion of $T$-injectivity ($T$-projectivity) is different from the notion of an $M$-injective ($M$-projective) module in \cite{fullerbook}.
\begin{cor}\label{2.3}
Let $R$ be a $T$-cocoherent ring and $G\in {\rm Gen}T$ be a module. Then the
following assertions are equivalent:
\begin{enumerate}
\item [\rm (1)]
$G$ is Gorenstein $\pi[T]$-projective;
\item [\rm (2)]
There is an exact sequence $0\rightarrow G \rightarrow B^{0}\rightarrow B^{1}\rightarrow\cdots $ of modules, where every $B^i$ is $\pi[T]$-projective;
\item [\rm (3)]
There is a short exact sequence $0\rightarrow G\rightarrow M\rightarrow I\rightarrow 0$ of modules, where
$M$ is $\pi[T]$-projective and $I $ is Gorenstein $\pi[T]$-projective.
\end{enumerate}
\end{cor}
\begin{proof}
{
$(1)\Longrightarrow (2)$ and $(1)\Longrightarrow (3)$ follow from definition.

$(2)\Longrightarrow (1)$ For  module $G\in {\rm Gen}T$, \cite[Proposition 2.1]{shaveisicam} implies that ${\rm Gen}T={\rm Pres}^\infty T$. So, there is an exact sequence
$$\cdots\longrightarrow T_{1}\longrightarrow T_{0}\longrightarrow G\longrightarrow 0 $$
where any $T_{i}$, $\pi[T]$-projective by Remark \ref{2.g}. Thus, the exact sequence
$$ \cdots\longrightarrow T_1\longrightarrow T_{0}\longrightarrow B^0\longrightarrow
B^1\longrightarrow\cdots$$
of $\pi[T]$-projective modules exists, where $G={\rm ker}(B^0\rightarrow B^1)$. Therefore, $G$ is Gorenstein $\pi[T]$-projective, by Theorem \ref{2.30}.

$(3)\Longrightarrow (2)$ Assume that the exact sequence
$$ 0\longrightarrow G\longrightarrow M \longrightarrow I\longrightarrow 0 \ \ (1)$$ exists, where
$M$ is $\pi[T]$-projective and $I $ is Gorenstein $\pi[T]$-projective. Since $I$ is Gorenstein $\pi[T]$-projective, there is an exact sequence
$$0\rightarrow I\rightarrow C^{0}\rightarrow C^{1}\rightarrow\cdots \ \ (2)$$
where every $C^{i}$ is $\pi[T]$-projective.
Assembling the sequences $(1)$ and $(2)$, we
get the exact sequence
$$0\rightarrow G\rightarrow M\rightarrow C^{0}\rightarrow C^{1}\rightarrow\cdots,$$ where $M$ and every $C^{i}$
is $\pi[T]$-projective, as desired.}
\end{proof}

\begin{prop}\label{2.5}
For any module $G\in{\rm Gen}T$, the following statements hold.
\begin{enumerate}
\item [\rm (1)]
If $G$ is Gorenstein $\pi[T]$-projective, then ${\mathcal{E}}_{T}^{i}(G, U)=0$
for all $i>0$ and every module $U\in {\rm F.Copres}^{1}T$ with ${\rm T.i.dim}(U) < \infty$.
\item [\rm (2)]
If $0\rightarrow N\rightarrow  G_{n-1}\rightarrow\cdots \rightarrow G_{0}\rightarrow G\rightarrow 0$ is an exact sequence of
modules where every $G_i$ is a Gorenstein $\pi[T]$-projective and $G_{i}\in{\rm Gen}T$, then ${\mathcal{E}}_{T}^{i}(N, U)={\mathcal{E}}_{T}^{n+i}(G, U)$ for any $i>0$ and any module $U\in {\rm F.Copres}^{1}T$ with ${\rm T.i.dim} (U) < \infty$.
\end{enumerate}
\end{prop}
\begin{proof}
{
(1) Let $G$ be a Gorenstein $\pi[T]$-projective module, and ${\rm T.i.dim}(U)=m<\infty$.
Then by hypothesis, the following $\pi[T]$-projective
resolution of $G$ exists:
$$0\rightarrow G\rightarrow B^{0}\rightarrow \cdots \rightarrow B^{m-1}\rightarrow N\rightarrow 0.$$
By Remark \ref{2.g}, ${\mathcal{E}}_{T}^{i}(B_j, U)= 0$ for every $i > 0$ and every $0 \leq j \leq m-1$. Since ${\rm T.i.dim} (U)=m$, we deduce that
 ${\mathcal{E}}_{T}^{i}(G,U)\cong{\mathcal{E}}_{T}^{m+i}(N, U)=0$.

(2) Setting $G_n=N$ and $K_j=\ker (G_{j}\rightarrow G_{j-1})$, for every $0\leq j\leq n$, the short exact sequence
$0\rightarrow K_j\rightarrow G_{j}\rightarrow K_{j-1}\rightarrow 0$ exist. Thus by (1), the induced exact sequences
$$0={\mathcal{E}}_T^r(G_{j},U)\rightarrow{\mathcal{E}}_T^r(K_{j},U)\rightarrow{\mathcal{E}}_T^{r+1}(K_{j-1},U)\rightarrow
{\mathcal{E}}_T^{r+1}(G_{j},U)=0$$
exists and so ${\mathcal{E}}_T^r(K_{j},U)\cong{\mathcal{E}}_T^{r+1}(K_{j-1},U)$, for every $r\geq 0$. Since $K_{n-1}=N$, we have
$${\mathcal{E}}_T^{n+i}(G,U)\cong{\mathcal{E}}_T^{n+i-1}(K_{0},U)\cong\cdots\cong {\mathcal{E}}_T^{i}(N,U),$$
as desired.}
\end{proof}
Next, we study the Gorenstein $\pi[T]$-projectivity of modules on $T$-cocoherent rings, in short exact sequences.
\begin{prop}\label{2.6}
Let $R$ be $T$-cocoherent and consider the exact sequence $0\rightarrow N\rightarrow B\rightarrow G\rightarrow 0$, where $B$ is $\pi[T]$-projective. Then ${\rm GT}$-${\rm pd}(G)\leq{\rm GT}$-${\rm pd}(N)+1$. In particular, if $G$ is Gorenstein $\pi[T]$-projective, so is $N$.
\end{prop}
\begin{proof}
{
We shall show that ${\rm GT}$-${\rm pd}(G)\leq{\rm GT}$-${\rm pd}(N) +1$. In fact, we may assume that
${\rm GT}$-${\rm pd}(N)=n< \infty$. Then, by definition, $N$ admits a Gorenstein $\pi[T]$-projective resolution:
$$0\rightarrow B_{n}\rightarrow B_{n-1}\rightarrow \cdots \rightarrow B_{0}\rightarrow N\rightarrow 0.$$
Assembling this sequence and the short exact sequence $0\rightarrow N\rightarrow B\rightarrow G\rightarrow 0$, the
following commutative diagram is obtained:

\begin{center}
$
\begin{array}{ccccccccccccccccc}
0 &\longrightarrow & B_{n} & \longrightarrow & \cdots& \longrightarrow& B_1\longrightarrow& B_0 & \longrightarrow & B&  \longrightarrow &G&  \longrightarrow & 0 \\
&  &  &   & & & & \downarrow & & \uparrow & & \\
 &  &  &   & & & & N& ={\hspace{-2mm}=}{\hspace{-2mm}={\hspace{-2mm}=}} & N &  \\
&  &  &   & & & & \downarrow & & \uparrow  & \\
&  &  &   & & & & 0 & & 0 &  \\
\end{array}
$
\end{center}
which shows that ${\rm GT}$-${\rm pd}(G)\leq n+1$. The particular case follows from Corollary \ref{2.3}.}
\end{proof}
\begin{prop}\label{2.7}
Let $R$ be a $T$-cocoherent ring and $0\rightarrow N\rightarrow G\rightarrow B\rightarrow 0$ be an exact sequence, where $N,B \in{\rm Gen}T$. If $N$ is Gorenstein $\pi[T]$-projective and $B$ is $\pi[T]$-projective, then $G$ is Gorenstein $\pi[T]$-projective.
\end{prop}
\begin{proof}
{Since $N$ is Gorenstein $\pi[T]$-projective, by Corollary $\ref{2.3}$, there exists an exact sequence of
$0\rightarrow N\rightarrow B^{'}\rightarrow K\rightarrow 0$, where $B^{'}$ is $\pi[T]$-projective and $K$ is Gorenstein $\pi[T]$-projective. 
Now, we consider the following diagram:
\begin{center}
$
\begin{array}{ccccccccc}

 & & 0 & & 0 & & \\
 & & \downarrow & & \downarrow & & \\
0 &\longrightarrow & N & \longrightarrow & G &\longrightarrow &B & \longrightarrow & 0 \\
& & \downarrow & & \downarrow & & \parallel&& \\
0 & \longrightarrow & B'& \longrightarrow & D& \longrightarrow & B & \longrightarrow & 0 \\
 & & \downarrow & & \downarrow & & \\
 & & K & ={\hspace{-1.5mm}=} & K & & \\
 & & \downarrow & & \downarrow & & \\
 & & 0 & & 0 & & \\
\end{array}
$
\end{center}
The exactness of the middle horizontal sequence with $B$ and $B^{'}$, $\pi[T]$-projective, implies that $D$ is $\pi[T]$-projective. Hence from the middle vertical sequence
and Corollary $\ref{2.3}$, we deduce that $G$ is Gorenstein $\pi[T]$-projective.
}
\end{proof}
%%%%%%%%%%%%%%%%%%%%%%%%%%%%%%%%%%%%%%%%%%%%%%%%%%%%%%%%%%%%%%%%%%%%%%%%%%%%%%%%%%%%%%%%%%%%%%%%%%%%%%%%%%%%%%%%%%%%%%%%%%%%%%%%%%%%%%%%%%%%%%%%%%%%%%%%%%%%%%%%%%%%%%%%%

\section{ Gorensetein $\pi[T]$-projective Modules on $T$-Cocoherent Rings}

This section is devoted to $T$-cocoherent rings over which every module is
Gorenstein $\pi[T]$-projective.
\begin{lem}\label{3.87}
Let $T$ be a tilting module and $G\in{Gen}T$. Then, $G\in{Cogen}T$.
\end{lem}
\begin{proof}
{
Let $G\in{Gen}T$. Then, the short exact sequence 
$0 \rightarrow K\rightarrow T^{(I)}\rightarrow G\rightarrow 0$ exists. We have  $K\subseteq T^{(I)}\subseteq T^I$. So, $K\in{\rm Cogen}T$. By \cite[Proposition 2.1]{shaveisicam}, ${\rm Cogen}T={\rm Copres}^\infty T$, since $T$ is tilting. Thus by Lemma \ref{2.56}, 
$G\in{\rm Copres}^m T$, and hence $G\in{Cogen}T$.

}
\end{proof}
\begin{prop}\label{3.1}
Let $R$ be a ring. The following assertions are equivalent:
\begin{enumerate}
\item [\rm (1)]
Every module belong ${\rm Gen}T$, is Gorenstein $\pi[T]$-projective;
\item [\rm (2)]
The ring satisfies the following two conditions:

{\rm (i)} Every $T$-injective module is $\pi[T]$-projective.

{\rm (ii)} ${\mathcal{E}}_T^{1}(N, U)=0$ for any $N\in{\rm Gen}T$ and any $U\in{\rm F.Copres}^nT$ with ${\rm T.i.dim}(U) <\infty$.
\end{enumerate}
\end{prop}
\begin{proof}
{
$(1)\Longrightarrow (2)$ The condition $(i)$ follows from this fact that every $T$-injective
module $M$ is Gorenstein $\pi[T]$-projective. So, the following $\pi[T]$-projective resolution of $M$ exists:
$$0 \rightarrow M\rightarrow B^0\rightarrow B^1\rightarrow\cdots.$$
Since $M$ is $T$-injective, $M$ is
$\pi[T]$-projective as a direct summand of $B^0$. Also,  Proposition \ref{2.5}(1) and (1) imply that ${\mathcal{E}}_{T}^{1}(N, U)=0$ for any module $N\in{\rm Gen}T$ and any module $U\in {\rm F.Copres}^{1}T$ with finite $T$-injective dimension.  So the condition $(ii)$ follows.

$(2)\Longrightarrow (1)$ Let $G\in{\rm Gen}T$ be.Then by Lemma \ref{3.87}, $G\in{Cogen}T$. So, a ${\rm Add}T$-resolution $\cdots \rightarrow T_{1}\rightarrow T_{0}\rightarrow G\rightarrow 0$ and a ${\rm Prod}T$- resolution $0\rightarrow G\rightarrow T^{0}\rightarrow T^{1}\rightarrow\cdots$ of $G$ exists. By Remark \ref{2.g}, any $T_i$ is $\pi[T]$-projective and any $T^i$ is $T$-injective. Hence by (2), every $T^i$ is  $\pi[T]$-projective.  Assembling these resolutions, we get the following exact sequence of $\pi[T]$-projective modules:
$${\mathbf{B}}=\cdots \rightarrow T_{1}\rightarrow T_{0}\rightarrow T^{0}\rightarrow T^{1}\rightarrow\cdots,$$
where $G={\rm ker}(T^0\rightarrow T^1)$. So by (2)(ii),
${\rm Hom}({\mathbf{B}},U)$ is exact for any module $U\in {\rm F.Copres}^{1}T$ with finite $T$-injective dimension. Hence $G$ is Gorenstein $\pi[T]$-projective.}
\end{proof}

The next theorem shows that if $R$ is a $T$-cocoherent ring and every $\sigma[T]$-injective module is Gorenstein $\pi[T]$-projective, then every  module is Gorenstein $\pi[T]$-projective.
\begin{thm}\label{3.2}
Let $R$ be a $T$-cocoherent ring. Then the following are equivalent:
\begin{enumerate}
\item [\rm (1)]
Every module is Gorenstein $\pi[T]$-projective;
\item [\rm (2)]
Every Gorenstein $\sigma[T]$-injective module is Gorenstein $\pi[T]$-projective;
\item [\rm (4)]
Every $\sigma[T]$-injective module is Gorenstein $\pi[T]$-projective;
\item [\rm (5)]
Every $T$-injective module is $\pi[T]$-projective.
\end{enumerate}
\end{thm}
\begin{proof}
{$(1)\Longrightarrow (2)$ This is a clear.

$(2)\Longrightarrow (4)$  Let $G$ be a $\sigma[T]$-injective module. Every $\sigma[T]$-injective module is Gorenstein $\sigma[T]$-injective (see,\cite{amini}). Since $G$ is Gorenstein $\sigma[T]$-injective, impelis that  $G$ is Gorenstein
$\pi[T]$-projective by hypothesis. 

$(4)\Longrightarrow (5)$  Let $G$ be a $T$-injective module. Since $G$ is $\sigma[T]$-injective, impelis that  $G$ is Gorenstein
$\pi[T]$-projective by hypothesis. By Corollary $\ref{2.3}$, there exists an exact sequence $0 \rightarrow G \rightarrow B \rightarrow N \rightarrow 0,$
where $B$ is $\pi[T]$-projective. Thus the sequence splits. Hence $G$ is
$\pi[T]$-projective as a direct summand of $B$.

$(5)\Longrightarrow (1)$
Let $G\in {\rm Gen}T$. Then by Lemma \ref{3.87},  there is an exact sequence 
$$0\longrightarrow G\longrightarrow T^0\longrightarrow T^{1}\longrightarrow\cdots $$
where any $T^{i}$ is $T$-injective. Then by (5), every $T^{i}$ is $\pi[T]$-projective. Hence Corollary $\ref{2.3}$ completes the proof.

}

\end{proof}

We denote the right $\pi[T]$-projective dimension of any module $M$ by $\pi[T].pd(M)$, and 
$\pi[T].pd(M)=\inf\{n:\ \mathcal{E}^{n+1}_T(M,N)=0\ {\rm for \ every} \ N\in\pi[T]\}.$

\begin{example}\label{34}
Let $R$ be a $1$-Gorenstein ring and
$0\rightarrow R\rightarrow E^{0}\rightarrow  E^{1} \rightarrow 0$
be  the minimal injective resolution of $R$. Then, $\pi[T].pd(E^0)=\pi[T].pd(E^1)=0$. Since by \cite{E.E}, $T=E_{0}\oplus E_{1}$ is a tilting
module. So, any $E^{i}$ is $\pi[T]$-projective and hence, any  $E^{i}$ is Gorenstein $\pi[T]$-projective for $i=0,1$.
\end{example}

\begin{definition}
We define the {\it global $\pi[T]$-projective dimension} of any ring $R$ to be:
$$gl.\pi[T].pd(R)=\sup \{\pi[T].pd(M)|\ M\ {\rm is\ a\ module}\}.$$
\end{definition}

Clearly, every $\pi[T]$-projective module is Gorenstein $\pi[T]$-projective. But the converse is not true in general. We finish this paper with the following theorem which determines a sufficient condition under which the converse holds.
\begin{thm}\label{3.4}
If $gl.\pi[T].pd(R)<\infty$, then every Gorenstein $\pi[T]$-projective module is $\pi[T]$-projective.
\end{thm}
\begin{proof}
{Suppose that $gl.\pi[T]$.${\rm pd}(R)=m <\infty$, and $G$ is a Gorenstein $\pi[T]$-projective
module. If $m = 0$, then $\mathcal{E}^{1}_T(M,N)=0$ for any $N\in\pi[T]$ , and hence $G$ is $\pi[T]$-projective. For $m\geq 1$, since $G$
is Gorenstein $\pi[T]$-projective, there exists an exact sequence $0\rightarrow G\rightarrow B^0\rightarrow B^1\rightarrow \cdots$ with each $B^i$ is $\pi[T]$-projective.
Let $L = {\rm coker}(B^{m-2}\rightarrow B^{m-1})$ . Then
$$0\longrightarrow G\longrightarrow B^0\longrightarrow B^1\longrightarrow\cdots\longrightarrow B^{m-2}\longrightarrow B^{m-1}\longrightarrow L\longrightarrow 0 $$
is exact, and hence $G$ is $\pi[T]$-projective since $\pi[T]$.${\rm pd}(L)\leq m$.}
\end{proof}

%\noindent {\bf Acknowledgment.} The authors would like to thank
%the referee for the valuable suggestions and comments.

\end{document}